\documentclass[11pt]{amsart}
\usepackage[T1]{fontenc}
\usepackage{lmodern}
\usepackage{amsmath,amssymb,amsthm,mathtools}
\usepackage{microtype}
\usepackage{needspace}
\usepackage[a4paper,left=27mm,right=27mm,top=26mm,bottom=27mm]{geometry}
\usepackage[hidelinks]{hyperref}
\hypersetup{pdftitle={Nonnegative curvature on vector bundles over CP3 and its blow-up},pdfauthor={Shen Wen}}
\numberwithin{equation}{section}
\newtheorem{theorem}{Theorem}[section]
\newtheorem{proposition}[theorem]{Proposition}
\newtheorem{lemma}[theorem]{Lemma}
\newtheorem{corollary}[theorem]{Corollary}
\theoremstyle{definition}

\theoremstyle{remark}

\newcommand{\CP}{\mathbb{C}P}
\newcommand{\HP}{\mathbb{H}P}
\newcommand{\R}{\mathbb{R}}
\newcommand{\C}{\mathbb{C}}
\newcommand{\Z}{\mathbb{Z}}
\newcommand{\Ztwo}{\mathbb{Z}/2}
\newcommand{\Zfour}{\mathbb{Z}/4}
\newcommand{\SO}{\operatorname{SO}}
\newcommand{\SU}{\operatorname{SU}}
\newcommand{\Sp}{\operatorname{Sp}}
\newcommand{\U}{\operatorname{U}}
\newcommand{\Tot}{\operatorname{Tot}}
\newcommand{\rank}{\operatorname{rank}}
\newcommand{\codim}{\operatorname{codim}}
\newcommand{\secc}{\operatorname{sec}}
\newcommand{\eps}{\varepsilon}

\newcommand{\PSq}{\mathfrak{P}}
\newcommand{\rhofourtwo}{\rho_2^{\,4}}
\newcommand{\Vect}{\operatorname{Vect}}

\title[Nonnegative curvature on vector bundles]
{Nonnegative curvature on vector bundles over $\CP^3$ and its blow-up}
\author{Shen Wen}
\date{}
\subjclass[2020]{53C20, 57R22, 57S15}
\keywords{Nonnegative sectional curvature, vector bundle, converse soul question, cohomogeneity one, Pontryagin class, blow-up}

\address{College of Mathematics and Physics, Wenzhou University, Wenzhou, P.R.China}
\email{shenwen121212@163.com}

\begin{document}
\begin{abstract}
Every real vector bundle of rank at least seven over $\CP^3$, and every real vector bundle of rank at least eight over $\CP^3\#\CP^3$, admits a complete metric of nonnegative sectional curvature on its total space. The corresponding unit sphere bundles also admit nonnegative sectional curvature. These are fixed-rank results: no additional trivial summand is required, and there is no restriction on any characteristic class. The proof combines low-dimensional bundle classification with the Grove--Ziller lifting theory over $S^4$ and $\CP^2$. 
\end{abstract}
\maketitle

\section{Introduction}

The Soul Theorem of Cheeger and Gromoll \cite{CG72} identifies every complete open manifold of nonnegative sectional curvature with the total space of the normal bundle of a compact soul. Its converse asks which vector bundles over closed nonnegatively curved manifolds admit such a metric on their total spaces. Even when the base is a classical manifold, the topology of the bundle need not make the construction of a metric apparent.

There are two distinct forms of this question. One can seek a metric on a given bundle $\xi$, or only on a stabilization
\[
\xi\oplus\eps^k,
\]
where $\eps^k$ is a trivial real bundle and $k$ is allowed to depend on $\xi$. Gonz\'alez-\'Alvaro \cite{GA17} proved the stabilized statement for every real or complex vector bundle over a compact rank-one symmetric space. Further stable converse-soul results were obtained for positively curved homogeneous spaces by Gonz\'alez-\'Alvaro and Zibrowius \cite{GAZ21}, and for classes of cohomogeneity-one manifolds by Amann, Gonz\'alez-\'Alvaro and Zibrowius \cite{AGZ22}. These results do not, by themselves, give a metric on the original bundle in a prescribed rank. In particular, topological cancellation of a trivial summand is not a curvature-preserving operation.

We obtain fixed-rank conclusions for two six-manifolds. The first is the complex projective space $\CP^3$. The second is its one-point blow-up, whose underlying smooth manifold is diffeomorphic to $\CP^3\#\CP^3$; the orientation convention is explained in Section~\ref{subsec:projective-model}.

\begin{theorem}\label{thm:cp3}
Let $\xi\to\CP^3$ be a real vector bundle of rank $m\ge7$. Then $\Tot(\xi)$ admits a complete Riemannian metric of nonnegative sectional curvature. Its unit sphere bundle $S(\xi)$ also admits a Riemannian metric of nonnegative sectional curvature.
\end{theorem}

\begin{theorem}\label{thm:blowup}
Let $\xi\to\CP^3\#\CP^3$ be a real vector bundle of rank $m\ge8$. Then $\Tot(\xi)$ admits a complete Riemannian metric of nonnegative sectional curvature. Its unit sphere bundle $S(\xi)$ also admits a Riemannian metric of nonnegative sectional curvature.
\end{theorem}

The bounds are independent of all characteristic classes. In particular, Theorem~\ref{thm:cp3} covers the full stable rank range over $\CP^3$, including arbitrarily negative first Pontryagin classes. Theorem~\ref{thm:blowup} covers every rank-eight bundle over the blow-up and hence, in particular, every rank-nine bundle. Neither theorem asserts that its rank bound is optimal.

The main geometric input is due to Grove and Ziller. They constructed nonnegatively curved metrics on all vector bundles over $S^4$ \cite{GZ00}, using commuting lifts of a cohomogeneity-one action. Their later work \cite{GZ11} established the corresponding lifting results over $\CP^2$, including all principal $\SO(r)$-bundles for $r\ge5$ and a precise description of the rank-four exceptions. We use these results together with their product-structure-group principle, not a new cohomogeneity-one metric construction.

The additional point is a bounded-rank realization of \emph{every} bundle class on the two six-dimensional bases. In ranks at least seven, the first Pontryagin class classifies real vector bundles on either base. For the twistor projection $\pi:\CP^3\to S^4$, the classes of even parity are represented by $\pi^*\eta$, while those of odd parity are represented by
\[
\lambda_{\R}\oplus\pi^*\eta,
\]
where $\lambda$ is the Hopf line bundle. Thus the only extra summand needed to pass from $S^4$ to $\CP^3$ is one real two-plane bundle.

For the blow-up projection $\pi:M\to\CP^2$, we show that every bundle in the asserted range has a model
\begin{equation}\label{eq:intro-model}
\xi\cong\bigoplus_{j=1}^{d}(L_j)_{\R}\oplus\pi^*\eta,
\qquad d\le2,
\end{equation}
with $L_j$ complex line bundles and with the frame bundle of $\eta$ admitting the required commuting lift. The last condition is essential in rank eight: two line summands leave a rank-four bundle over $\CP^2$, and not every such frame bundle has a lift. A choice of the line-bundle weights modulo four avoids exactly the residual Pontryagin classes that would cause this difficulty. This is the refinement that yields the bound eight rather than nine.

The models in \eqref{eq:intro-model} are not used to infer that pullbacks or Whitney sums preserve nonnegative curvature. Instead, a fiber product combines the auxiliary principal bundles over the four-dimensional base. After taking a further sphere factor in the blow-up case, all summands are realized simultaneously by one orthogonal representation of a compact structure group. O'Neill's submersion formula then gives the desired metrics. We make the common acting group and the associated-bundle identifications explicit, so that no equivariant cancellation or unproved descent of a lifted action is involved.

\medskip
\noindent\emph{Conventions.}
All manifolds, group actions and bundles used in the metric constructions are smooth. Bundle isomorphisms cover the identity of the specified base. The notation $E_{\R}$ denotes the realification of a complex bundle $E$, and $\eps^r$ denotes a trivial real rank-$r$ bundle. The sphere bundle is formed using a fiber inner product; different choices give smoothly isomorphic sphere bundles. The existence assertions below do not prescribe a metric on the base.

\section{Classification by the first Pontryagin class}\label{sec:classification}

We record explicitly why the first Pontryagin class is sufficient in the ranks under consideration. This step concerns actual vector bundles, not only stable equivalence classes.

Write $\rho_j$ for reduction of integral cohomology modulo $j$, write
\[
i_*:H^*(-;\Ztwo)\longrightarrow H^*(-;\Zfour)
\]
for the map induced by $1\mapsto2$, and write $\rhofourtwo:H^*(-;\Zfour)\to H^*(-;\Ztwo)$ for the coefficient reduction induced by $\Z/4\to\Z/2$. The Pontryagin square
\[
\PSq:H^2(-;\Ztwo)\longrightarrow H^4(-;\Zfour)
\]
satisfies
\begin{equation}\label{eq:pontryagin-square}
\rhofourtwo\PSq(a)=a^2,
\qquad
\PSq(\rho_2\ell)=\rho_4(\ell^2).
\end{equation}

We use \v{C}adek--Van\v{z}ura \cite[Proposition~2]{CV92}. For a connected CW complex $X$ of dimension at most seven, the characteristic-class map
\[
[X,B\SO(7)]\longrightarrow
H^2(X;\Ztwo)\oplus H^4(X;\Ztwo)\oplus H^4(X;\Z),
\quad
\zeta\longmapsto(w_2(\zeta),w_4(\zeta),p_1(\zeta))
\]
has image the triples $(a,b,c)$ satisfying
\begin{equation}\label{eq:CV}
\rho_4(c)=\PSq(a)+i_*(b).
\end{equation}
It is injective if $H^4(X;\Z)$ has no element of order four.

\begin{proposition}\label{prop:classification}
Let $X$ be a simply connected finite CW complex of dimension at most six. Suppose that $H^3(X;\Ztwo)=0$, that $H^4(X;\Z)$ is torsion-free, and that squaring induces an isomorphism
\begin{equation}\label{eq:squaring}
H^2(X;\Ztwo)\longrightarrow H^4(X;\Ztwo),
\qquad a\longmapsto a^2.
\end{equation}
For each $m\ge7$, the first Pontryagin class gives a bijection
\[
p_1:\Vect^m_{\R}(X)\longrightarrow H^4(X;\Z),
\]
where $\Vect^m_{\R}(X)$ denotes isomorphism classes of real rank-$m$ bundles.
\end{proposition}
\begin{proof}
Since $X$ is simply connected, every real vector bundle over $X$
is orientable. We first consider rank seven.

\smallskip
\noindent
\emph{Surjectivity.}
Let $c\in H^4(X;\Z)$ be arbitrary. Since the squaring map
\[
H^2(X;\Z/2)\longrightarrow H^4(X;\Z/2),
\qquad a\longmapsto a^2,
\]
is an isomorphism, there is a unique class
$a\in H^2(X;\Z/2)$ satisfying
\[
a^2=\rho_2(c).
\]
The defining property of the Pontryagin square gives
\[
\rho_2^4\bigl(\mathfrak P(a)\bigr)=a^2,
\]
and hence
\[
\rho_2^4\bigl(\rho_4(c)-\mathfrak P(a)\bigr)=0.
\]
By exactness of the coefficient sequence
\[
0\longrightarrow\Z/2
\xrightarrow{i}\Z/4
\longrightarrow\Z/2
\longrightarrow0,
\]
there exists
$b\in H^4(X;\Z/2)$ such that
\[
\rho_4(c)=\mathfrak P(a)+i_*(b).
\]
Thus $(a,b,c)$ satisfies the characteristic-class relation in the
theorem of \v{C}adek--Van\v{z}ura. Its existence part therefore
produces an oriented rank-seven bundle $\zeta_c$ with
\[
w_2(\zeta_c)=a,\qquad
w_4(\zeta_c)=b,\qquad
p_1(\zeta_c)=c.
\]
Since $c$ was arbitrary, $p_1$ is surjective in rank seven.

\smallskip
\noindent
\emph{Injectivity.}
Suppose that $\zeta$ and $\zeta'$ are oriented rank-seven bundles
with
\[
p_1(\zeta)=p_1(\zeta')=c.
\]
Reducing the characteristic-class relation modulo two gives
\[
w_2(\zeta)^2=\rho_2(c)=w_2(\zeta')^2.
\]
Since squaring is injective,
\[
w_2(\zeta)=w_2(\zeta').
\]
The modulo-four relation then gives
\[
i_*w_4(\zeta)=i_*w_4(\zeta').
\]
Because $H^3(X;\Z/2)=0$, the coefficient exact sequence shows that
$i_*$ is injective in degree four. Hence
\[
w_4(\zeta)=w_4(\zeta').
\]
The injectivity part of the \v{C}adek--Van\v{z}ura theorem now gives
$\zeta\cong\zeta'$.

Finally let $m>7$. Every rank-$m$ bundle over the
six-dimensional complex $X$ splits as
\[
\xi\cong\xi_7\oplus\varepsilon^{m-7},
\]
because a real bundle of rank greater than $\dim X$ has a
nowhere-zero section. Moreover, $p_1(\xi)=p_1(\xi_7)$. Therefore, the first Pontryagin class determines the isomorphism classes of  rank-$m$ bundles over $X$.
\end{proof}


We will also use the elementary classification of complex rank-two bundles over a four-dimensional CW complex. The universal Chern classes define a map
\[
(c_1,c_2):BU(2)\longrightarrow K(\Z,2)\times K(\Z,4).
\]
In degrees at most four,
\[
\pi_i(BU(2))=
\begin{cases}
\Z,&i=2,4,\\
0,&i=1,3,
\end{cases}
\]
and the two Chern classes induce the corresponding isomorphisms: $c_1$ does so on $\pi_2$, while a generator $S^4\to BU(2)$ classifies a rank-two bundle with $c_2=\pm1$, so $c_2$ does so on $\pi_4$. The target has trivial $\pi_5$, hence $(c_1,c_2)$ is $5$-connected. Obstruction theory therefore gives, for every CW complex $Y$ of dimension at most four,
\[
[Y,BU(2)]\xrightarrow{\ \cong\ }
H^2(Y;\Z)\times H^4(Y;\Z),
\qquad [E]\longmapsto(c_1(E),c_2(E)).
\]
Thus every prescribed pair $(c_1,c_2)$ is realized by a complex rank-two bundle. We shall use the standard realification formula
\begin{equation}\label{eq:realification}
p_1(E_{\R})=c_1(E)^2-2c_2(E)
\end{equation}
and the Whitney sum formula for $p_1$ \cite{MS74}.

\Needspace{8\baselineskip}
\begin{lemma}\label{lem:four-dimensional-existence}
The following existence statements hold.
\begin{enumerate}
\item If $u$ generates $H^4(S^4;\Z)$, then for each $a\in\Z$ and $r\ge4$ there is an oriented real rank-$r$ bundle $\eta\to S^4$ with $p_1(\eta)=2au$.
\item If $h$ generates $H^2(\CP^2;\Z)$, then for each $C\in\Z$ and $r\ge4$ there is an oriented real rank-$r$ bundle $\eta_C\to\CP^2$ with $p_1(\eta_C)=Ch^2$.
\end{enumerate}
\end{lemma}

\begin{proof}
For the first assertion, take a complex rank-two bundle with $c_1=0$ and $c_2=-au$, realify it, and add $r-4$ trivial real lines. For the second, let $\delta\in\{0,1\}$ be congruent to $C$ modulo two and choose a complex rank-two bundle $E_C$ with
\begin{equation}\label{eq:EC}
c_1(E_C)=\delta h,
\qquad
c_2(E_C)=\frac{\delta-C}{2}h^2.
\end{equation}
The second coefficient is integral. Since $\delta^2=\delta$, formula~\eqref{eq:realification} gives $p_1((E_C)_{\R})=Ch^2$. Set $\eta_C=(E_C)_{\R}\oplus\eps^{r-4}$.
\end{proof}

\section{Commuting lifts and associated metrics}\label{sec:geometry}

\subsection{Quotients and curvature}
For a principal right $K$-bundle $P$ and an orthogonal representation $\rho:K\to O(V)$, our associated-bundle convention is
\begin{equation}\label{eq:associated-convention}
P\times_KV=(P\times V)/\!\sim,
\qquad
(p,v)\sim(pk,\rho(k)^{-1}v).
\end{equation}

\begin{lemma}\label{lem:quotient-curvature}
Let a compact Lie group $K$ act freely and isometrically on a compact Riemannian manifold $P$ with $\secc\ge0$. For every orthogonal $K$-representation $V$, the associated total space $P\times_KV$ admits a complete metric with $\secc\ge0$. The space $P\times_KS(V)$ also admits a metric with $\secc\ge0$ where $S(V)\subset V$ is the unit sphere.
\end{lemma}

\begin{proof}
Equip $P\times V$ with the product of the given metric and the Euclidean metric. The diagonal action from \eqref{eq:associated-convention} is free and isometric. The quotient is a Riemannian submersion. For orthonormal horizontal vectors $X,Y$, O'Neill's formula \cite{ON66} reads
\[
K_{\mathrm{quot}}(d\pi X,d\pi Y)
=K_{P\times V}(X,Y)+3\lVert A_XY\rVert^2\ge0.
\]
The product is complete, and its quotient by a compact group of isometries is complete. For the second assertion, apply the same argument to $P\times S(V)$, where $S(V)\subset V$ carries the standard metric induced by the Euclidean inner product on $V$ (so $\secc\equiv1$ on $S(V)$).
\end{proof}


\subsection{The product principle}
A \emph{commuting lift} of a $G$-action on $B$ to a principal right $L$-bundle $p:P\to B$ is a $G$-action on $P$ satisfying
\[
p(gq)=g\,p(q),\qquad g(q\ell)=(gq)\ell.
\]
A compact group action has cohomogeneity one when its orbit space is one-dimensional. We use actions whose orbit space is an interval and whose two singular orbits have codimension two. Grove--Ziller \cite[Theorem~E]{GZ00} proved that a compact cohomogeneity-one manifold with these orbit codimensions admits an invariant metric of nonnegative sectional curvature. Ineffective kernels cause no problem: one may pass to the effective quotient without changing the orbits or invariant metrics.

\begin{lemma}[Grove--Ziller product principle]\label{lem:fiber-product}
Let $B$ carry a cohomogeneity-one action of a compact connected group $G$ with two singular orbits of codimension two. Let $P_i\to B$ be principal bundles with compact connected structure groups $L_i$, and suppose that the same $G$-action admits commuting lifts to both bundles. Then
\[
Q=P_1\times_B P_2
\]
admits a $G\times L_1\times L_2$-invariant metric with nonnegative sectional curvature.
\end{lemma}

\begin{proof}
This is the direct-product construction in \cite[Lemma~1.6(c)]{GZ11}, together with the codimension-two metric theorem. We include the orbit calculation. The fiber product is a principal $L_1\times L_2$-bundle, and the joint left action is
\[
(g,\ell_1,\ell_2)(q_1,q_2)
=(gq_1\ell_1^{-1},gq_2\ell_2^{-1}).
\]
If $\mathcal O$ is a $G$-orbit in $B$, its inverse image is a single joint orbit: first move between the base points using $G$, and then move within the fiber using $L_1\times L_2$. Hence $Q/(G\times L_1\times L_2)=B/G$. Moreover,
\[
\begin{aligned}
\codim_Q p^{-1}(\mathcal O)
&=(\dim B+\dim L_1+\dim L_2)
 -(\dim\mathcal O+\dim L_1+\dim L_2)\\
&=\codim_B\mathcal O.
\end{aligned}
\]
The two singular orbits in $Q$ therefore have codimension two, so \cite[Theorem~E]{GZ00} applies. In particular, the resulting metric is invariant under the full right structure-group action.
\end{proof}

We will restrict this invariant metric to closed subgroups of the structure group and then apply Lemma~\ref{lem:quotient-curvature}. 

\subsection{A common acting group}\label{subsec:lifts}
The finite covering of the acting group must be distinguished from a finite covering of the structure group. Throughout both applications we use
\begin{equation}\label{eq:acting-group}
G=\SU(2),\qquad \nu:G\longrightarrow\SO(3),
\end{equation}
where $\nu$ is the standard double covering. The actions on the four-dimensional bases factor through $\nu$; their lifts need not do so. This is consistent with the covering convention preceding \cite[Lemma~1.6]{GZ11}.

On $S^4$, identify $\R^5$ with the traceless symmetric real $3\times3$ matrices and use the action $X\mapsto\nu(g)X\nu(g)^{-1}$ on the unit sphere. Generic matrices have finite stabilizer; matrices at the two ends of the orbit interval have a repeated eigenvalue. The corresponding singular orbits are two-dimensional.

On $\CP^2$, use
\begin{equation}\label{eq:CP2-action}
g\cdot[z]=[\nu(g)z],\qquad z\in\C^3\setminus\{0\}.
\end{equation}
Its singular orbits are the real locus $\R P^2$ and the conic $\{[z]:z_1^2+z_2^2+z_3^2=0\}$, both of real dimension two \cite[Section~2]{GZ11}. Thus both base actions have cohomogeneity one and singular-orbit codimensions two. The action \eqref{eq:CP2-action} is \emph{not} the block $\SU(2)$-action on $\C\oplus\C^2$, which has a fixed point in $\CP^2$.

\begin{proposition}\label{prop:lifts}
For the $G$-actions just specified:
\begin{enumerate}
\item the action on $S^4$ has a commuting lift to every principal $\SO(r)$-bundle for $r\ge4$, and to the quaternionic Hopf bundle $\Sp(1)\to S^7\to S^4$;
\item the action on $\CP^2$ has a commuting lift to every principal $\SO(r)$-bundle for $r\ge5$, and to the Hopf bundle $S^1\to S^5\to\CP^2$.
\end{enumerate}
\end{proposition}

\begin{proof}
For $S^4$, the rank-four assertion is \cite[Theorem~F]{GZ00}, pulled back to $G$ if necessary. A higher-rank real bundle over a four-dimensional base splits as $E_4\oplus\eps^{r-4}$. Its frame bundle is an extension
\[
P_4\times_{\SO(4)}\SO(r).
\]
A commuting lift extends by $g[p,A]=[gp,A]$, as in \cite[Lemma~1.6(a)]{GZ11}.

For the quaternionic Hopf bundle, put $P_0=S^7$ and $\overline P_0=P_0/\{\pm1\}$. The latter is a principal $\SO(3)$-bundle over $S^4$. \cite[Theorem~F]{GZ00} gives a commuting lift to $\overline P_0$, which we regard as a $G$-action. Because $G$ is simply connected, the map
\[
G\times P_0\longrightarrow\overline P_0,
\qquad (g,q)\longmapsto g\cdot\overline q
\]
lifts uniquely to $P_0$ after fixing its value at $(1,q_0)$. Its restriction to $\{1\}\times P_0$ is the identity. The action law follows from uniqueness of lifts on the connected space $G\times G\times P_0$. The two maps expressing commutation with right $\Sp(1)$-multiplication are likewise lifts of the same map to $\overline P_0$; they agree at the identity and hence agree everywhere. This yields the required commuting $G$-action on $P_0$. It is the special case of \cite[Lemma~1.6(b)]{GZ11} needed here. A linear realization is also described in \cite[Remark~4.6]{GZ00}. No descent of this lift to $\SO(3)$ is asserted or needed.

For $\CP^2$, the frame-bundle assertion is \cite[Corollary~4.11(c)]{GZ11}, with its standing convention that the acting group may be replaced by its simply connected cover. The corollary is stated after excluding torus reductions, which already admit lifts. In the present case this last assertion is direct: every line bundle over $\CP^2$ is a power of the Hopf line bundle, and $G$ acts linearly on $S^5\subset\C^3$ through $\nu$, commuting with the Hopf circle. The induced lifts to powers, products and extensions of structure group handle torus reductions. The same linear action is the claimed lift to $S^5$ itself.
\end{proof}

\subsection{The pullback identification}
The following elementary fact keeps the bundle identification separate from the metric construction.

\begin{lemma}\label{lem:pullback}
Let $P\to B$ and $R\to B$ be principal right $H$- and $L$-bundles, respectively, let $T\subset H$ be closed, and put $Q=P\times_B R$. Let $V$ be an $L$-representation, with $T$ acting trivially, and set $K=T\times L$. Then $Q/K\cong P/T$, and over this base there is a canonical vector-bundle isomorphism
\begin{equation}\label{eq:pullback-identification}
Q\times_KV\cong(P/T)\times_B(R\times_LV)=\pi^*(R\times_LV),
\end{equation}
where $\pi:P/T\to B$ is the induced projection. If $W$ is a $T$-representation on which $L$ acts trivially, then also
\begin{equation}\label{eq:W-identification}
Q\times_KW\cong P\times_TW
\end{equation}
as bundles over $P/T$.
\end{lemma}

\begin{proof}
First quotient $Q$ by $L$. This removes the $R$-coordinate, since $L$ acts transitively on each of its fibers, and gives $P$. Quotienting by $T$ then gives $P/T$.

For \eqref{eq:pullback-identification}, define
\[
\Phi\bigl([((q,p),v)]\bigr)=([q],[p,v]).
\]
The condition $(q,p)\in P\times_B R$ says that the two components on the right lie over the same point of $B$. Under a change of representative by $(t,\ell)\in T\times L$,
\[
((q,p),v)\longmapsto((qt,p\ell),\ell^{-1}v),
\]
we have $[qt]=[q]$ and $[p\ell,\ell^{-1}v]=[p,v]$. Thus $\Phi$ is well defined. Conversely, choose representatives $q$ and $(p,v)$ of a point $([q],[p,v])$ in the pullback and send it to $[((q,p),v)]$. The pullback condition guarantees $(q,p)\in Q$, and changing either representative changes the resulting triple by exactly the $K$-relation. This defines an inverse to $\Phi$. In local trivializations the maps are smooth and fiberwise linear.

For \eqref{eq:W-identification}, the map is $[((q,p),w)]\mapsto[q,w]$; quotienting first by $L$ proves that it is a smooth bundle isomorphism. Associated bundles preserve direct sums, so these identifications can be used simultaneously for $W\oplus V$.
\end{proof}

\section{Bundles over \texorpdfstring{$\CP^3$}{CP3}}\label{sec:cp3}

\subsection{The twistor fibration and the relevant characteristic classes}
Let
\[
\Sp(1)\longrightarrow P_0=S^7\xrightarrow{p_0}S^4=\HP^1
\]
be the quaternionic Hopf principal bundle, and fix the standard circle $T=\U(1)\subset\Sp(1)$. After identifying $\mathbb H^2$ with $\C^4$ by right multiplication by $i$, the restricted $T$-action on $S^7$ is the usual complex Hopf action. Hence
\[
P_0/T\cong\CP^3.
\]
The quotient map is the twistor fibration
\begin{equation}\label{eq:twistor}
S^2=\Sp(1)/T\longrightarrow\CP^3\xrightarrow{\pi}S^4.
\end{equation}

Let $\lambda=P_0\times_T\C$ be the complex line bundle associated to the weight-one character of $T$, and put $x=c_1(\lambda)$. Depending on the convention for the Hopf action, $\lambda$ is the tautological line bundle or its dual; in either case $x$ is a generator up to sign and
\begin{equation}\label{eq:cp3-cohomology}
H^*(\CP^3;\Z)=\Z[x]/(x^4),
\qquad
p_1(\lambda_{\R})=c_1(\lambda)^2=x^2.
\end{equation}

The cohomology ring immediately verifies the hypotheses of Proposition~\ref{prop:classification}: $H^3(\CP^3;\Ztwo)=0$, $H^4(\CP^3;\Z)$ is free, and squaring sends the generator of $H^2(\CP^3;\Ztwo)$ to the generator of $H^4(\CP^3;\Ztwo)$. We therefore have
\begin{equation}\label{eq:cp3-classification}
\Vect^m_{\R}(\CP^3)\xrightarrow[\cong]{\ p_1\ }H^4(\CP^3;\Z)=\Z\langle x^2\rangle,
\qquad m\ge7.
\end{equation}
Thus, in these ranks, it is enough to construct one nonnegatively curved bundle for each integer multiple of $x^2$.

\begin{lemma}\label{lem:twistor-cohomology}
The pullback
\[
\pi^*:H^4(S^4;\Z)\longrightarrow H^4(\CP^3;\Z)
\]
is an isomorphism. Hence a generator $u\in H^4(S^4;\Z)$ may be chosen so that $\pi^*u=x^2$.
\end{lemma}

\begin{proof}
Use the integral Serre spectral sequence of \eqref{eq:twistor}. Since the fiber $S^2$ has cohomology only in degrees $0$ and $2$, and the base $S^4$ only in degrees $0$ and $4$, the only nonzero $E_2$-term of total degree four is
\[
E_2^{4,0}=H^4(S^4;\Z)\cong\Z.
\]
No differential can enter or leave this term. The edge homomorphism is therefore an isomorphism. Reversing the sign of $u$ if necessary gives $\pi^*u=x^2$.
\end{proof}

\subsection{Proof of Theorem~\ref{thm:cp3}}
Let $\xi\to\CP^3$ have rank $m\ge7$, and write
\[
p_1(\xi)=kx^2,
\qquad k\in\Z.
\]
We treat the two parities of $k$ separately. This makes transparent why the Hopf line bundle is needed only in the odd case.

\medskip
\noindent\emph{Case 1: $k=2a$ is even.}
By Lemma~\ref{lem:four-dimensional-existence}, choose an oriented real rank-$m$ bundle
\[
\eta_a\longrightarrow S^4,
\qquad
p_1(\eta_a)=2au.
\]
Let $P_a\to S^4$ be its principal $\SO(m)$ frame bundle. Proposition~\ref{prop:lifts} gives commuting lifts of the same $G=\SU(2)$-action on $S^4$ to both $P_0$ and $P_a$. Form the fiber product
\[
Q_a=P_0\times_{S^4}P_a.
\]
By the Grove--Ziller product principle, Lemma~\ref{lem:fiber-product}, $Q_a$ carries a nonnegatively curved metric invariant under the right action of $\Sp(1)\times\SO(m)$.

Now set
\[
K=T\times\SO(m)\subset\Sp(1)\times\SO(m)
\]
and let $T$ act trivially on $\R^m$. Quotienting $Q_a$ first by $\SO(m)$ and then by $T$ gives
\[
Q_a/K\cong P_0/T\cong\CP^3.
\]
The quotient-in-stages Lemma~\ref{lem:pullback} gives a canonical bundle isomorphism
\begin{equation}\label{eq:cp3-even-bundle}
Q_a\times_K\R^m\cong\pi^*\eta_a.
\end{equation}
Consequently
\[
p_1(Q_a\times_K\R^m)
=\pi^*p_1(\eta_a)
=2a\,x^2
=kx^2
=p_1(\xi).
\]
By \eqref{eq:cp3-classification}, the bundle in \eqref{eq:cp3-even-bundle} is isomorphic to $\xi$. Lemma~\ref{lem:quotient-curvature} now gives the required complete nonnegatively curved metric on $\Tot(\xi)$ and the corresponding metric on $S(\xi)$.

\medskip
\noindent\emph{Case 2: $k=2a+1$ is odd.}
Here the pullback of a bundle from $S^4$ can supply only the even part of $p_1$. The missing copy of $x^2$ is supplied by $\lambda_{\R}$. Since $m\ge7$, we have $m-2\ge5$. Choose an oriented real rank-$(m-2)$ bundle
\[
\eta_a\longrightarrow S^4,
\qquad
p_1(\eta_a)=2au=(k-1)u,
\]
and let $P_a\to S^4$ be its frame bundle. Again put
\[
Q_a=P_0\times_{S^4}P_a,
\qquad
K=T\times\SO(m-2).
\]
The same lifting argument gives an invariant metric with $\secc\ge0$ on $Q_a$.

Let $K$ act orthogonally on
\[
V=\C_{\R}\oplus\R^{m-2}
\]
by
\[
(t,A)\cdot(w,v)=(tw,Av).
\]
The $\SO(m-2)$ factor acts trivially on the complex line, while $T$ acts trivially on $\R^{m-2}$. Therefore Lemma~\ref{lem:pullback}, applied to the two summands, gives
\begin{equation}\label{eq:cp3-odd-bundle}
Q_a\times_KV
\cong
\lambda_{\R}\oplus\pi^*\eta_a.
\end{equation}
Its first Pontryagin class is
\[
p_1(\lambda_{\R})+\pi^*p_1(\eta_a)
=x^2+(k-1)x^2
=kx^2
=p_1(\xi).
\]
Hence \eqref{eq:cp3-classification} identifies \eqref{eq:cp3-odd-bundle} with $\xi$. Applying Lemma~\ref{lem:quotient-curvature} to $V$, and separately to its unit sphere $S(V)$ with the standard round metric, completes the proof of Theorem~\ref{thm:cp3}.

\section{Bundles over the blow-up}\label{sec:blowup}

Set
$M=\CP^3\#\CP^3.$
The proof has four topological ingredients, which we keep separate: the topology for $M$, a principal $T^2$-bundle that realizes all line bundles, an arithmetic decomposition of $p_1$, and a liftable residual bundle over $\CP^2$. We then combine these data in one associated-bundle construction.

\subsection{The $S^2$-bundle structure and cohomology of $M$}\label{subsec:projective-model}
We use the description of Escher--Ziller \cite[Sections~4 and~5]{EZ14}.
They consider a family of $S^2$-bundles
\[
S^2\longrightarrow N_t
\xrightarrow{\;\pi\;}
\CP^2 .
\]
For $t=0$, they identify
\[
N_0\cong \CP^3\#\overline{\CP^3}
\]
(see \cite[Proposition~5.5]{EZ14}).  Since complex conjugation reverses orientation on $\CP^3$, the underlying unoriented smooth manifold of $N_0$ is diffeomorphic to $\CP^3\#\CP^3$.  We therefore identify $M$ with $N_0$ and use the resulting $S^2$-bundle
\[
S^2\longrightarrow M\xrightarrow{\;\pi\;}\CP^2.
\]

Let
$x\in H^2(M;\Z)$
denote the pullback of the standard generator of
$H^2(\CP^2;\Z)$, and let
$y\in H^2(M;\Z)$
be the second generator. The
cohomology ring of $M$ satisfies
\begin{equation}\label{eq:ring}
H^*(M;\Z)
\cong
\Z[x,y]/(x^3,\;y^2+xy),
\qquad |x|=|y|=2.
\end{equation}
In particular,
\[
H^2(M;\Z)=\Z\{x,y\},
\qquad
H^4(M;\Z)=\Z\{x^2,xy\}.
\]

Reducing \eqref{eq:ring} modulo two gives
\[
y^2=xy.
\]
Therefore
\[
(\alpha x+\beta y)^2
=
\alpha x^2+\beta xy,
\qquad
\alpha,\beta\in\Z/2,
\]
and hence the squaring map
\[
H^2(M;\Z/2)\longrightarrow H^4(M;\Z/2),
\qquad z\longmapsto z^2,
\]
is an isomorphism. Proposition~\ref{prop:classification} now gives
\begin{corollary}\label{cor:M-classification}
For every $m\ge7$, the first Pontryagin class induces a bijection
\[
p_1:\Vect^m_{\R}(M)\xrightarrow{\ \cong\ }H^4(M;\Z)
=\Z\{x^2,x y\}.
\]
Thus a rank-$m$ bundle is determined by the two integers $A,B$ in
\begin{equation}\label{eq:p1AB}
p_1(\xi)=A x^2+B\,xy.
\end{equation}
\end{corollary}

\subsection{A principal torus bundle over
\texorpdfstring{$M$}{M}}\label{subsec:torus}

We use the principal torus bundle constructed by Escher--Ziller.  In their notation,
\[
T^2\longrightarrow P_t\longrightarrow N_t=P_t/T^2
\]
is a principal $T^2$-bundle; see \cite[Proposition~5.1 and the remark following it]{EZ14}.  For $t=0$ they identify
\[
P_0\cong S^5\times S^3,
\qquad
N_0\cong\CP^3\#\overline{\CP^3};
\]
see \cite[Proposition~5.5]{EZ14}.  Using the smooth identification of the underlying unoriented manifold $N_0$ with $M=\CP^3\#\CP^3$ fixed above, we obtain a principal bundle
\begin{equation}\label{eq:torus-base}
T^2\longrightarrow Z:=S^5\times S^3\xrightarrow{\;p\;}M.
\end{equation}
The composite $Z\xrightarrow{p}M\xrightarrow{\pi}\CP^2$ is, under the identification of \cite[Proposition~5.5]{EZ14}, the Hopf projection on the $S^5$-factor.  We shall use only these structural facts about the torus bundle; no choice of coordinates on $T^2$ will be needed.

\subsection{Characters and line bundles}\label{subsec:characters}
Every character $\chi:T^2\to S^1$ determines a complex line bundle
\[
L_\chi=Z\times_{T^2}\C_\chi\longrightarrow M.
\]
The next elementary observation allows us to choose characters directly from their first Chern classes.

\begin{lemma}\label{lem:characters}
The map
\begin{equation}\label{eq:character-c1-isomorphism}
\operatorname{Hom}(T^2,S^1)\longrightarrow H^2(M;\Z),
\qquad
\chi\longmapsto c_1(L_\chi),
\end{equation}
is an isomorphism.  Hence, for every $r,\ell\in\Z$, there is a unique character $\chi_{r,\ell}$ such that the associated line bundle $L_{r,\ell}:=L_{\chi_{r,\ell}}$ satisfies
\begin{equation}\label{eq:line-c1}
c_1(L_{r,\ell})=r x+\ell y.
\end{equation}
In particular, every complex line bundle over $M$ is obtained from a unique character of $T^2$.
\end{lemma}

\begin{proof}
The character lattice $\operatorname{Hom}(T^2,S^1)$ is naturally identified with $H^1(T^2;\Z)$.  Since $Z=S^5\times S^3$ has
\[
H^1(Z;\Z)=H^2(Z;\Z)=0,
\]
the low-degree part of the Serre spectral sequence for the principal bundle \eqref{eq:torus-base} gives an isomorphism
\[
d_2:H^1(T^2;\Z)\xrightarrow{\ \cong\ }H^2(M;\Z).
\]
The first Chern class of the line bundle associated to a character is the corresponding transgression (up to the harmless sign determined by the associated-bundle convention).  Thus \eqref{eq:character-c1-isomorphism} is an isomorphism; compare the transgression calculation in \cite[Proposition~5.1 and the remark following it]{EZ14}.  Since $x,y$ form an integral basis of $H^2(M;\Z)$, the classes $rx+\ell y$ determine unique characters $\chi_{r,\ell}$.
\end{proof}

For later use, the realification of $L_{r,\ell}$ satisfies, by \eqref{eq:ring},
\begin{equation}\label{eq:line-p1}
\begin{aligned}
p_1((L_{r,\ell})_{\R})
&=c_1(L_{r,\ell})^2\\
&=(rx+\ell y)^2\\
&=r^2x^2+(2r\ell-\ell^2)xy.
\end{aligned}
\end{equation}
The mixed coefficient $2r\ell-\ell^2$ is the reason that line bundles can be used to prescribe the $xy$-component of $p_1$.

\subsection{Choosing the line bundles}\label{subsec:decomposition}
Let $\xi\to M$ have rank $m\ge8$, and write $p_1(\xi)$ as in \eqref{eq:p1AB}. We first choose one or two line bundles so that their realifications account for the entire $B\,xy$ term. What remains is a multiple $C x^2$, which can be pulled back from $\CP^2$. If two line bundles are needed and $m=8$, the residual bundle has rank four; we must then arrange that $C$ lies in the part of the rank-four Grove--Ziller lifting range that we can realize.

\begin{lemma}\label{lem:arithmetic}
For every $A,B\in\Z$ there exist $d\in\{1,2\}$ and integers $(r_i,\ell_i)$, $1\le i\le d$, such that
\begin{equation}\label{eq:arithmetic}
A x^2+B\,xy
=\sum_{i=1}^{d}(r_i x+\ell_i y)^2+C x^2
\end{equation}
for some $C\in\Z$. The choices can be made so that
\begin{equation}\label{eq:C-condition}
d=2\quad\Longrightarrow\quad C\not\equiv2\pmod4.
\end{equation}
\end{lemma}

\begin{proof}
We choose the weights according to the residue class of $B$.

If $B$ is odd, one line bundle is enough: take
\[
(r_1,\ell_1)=\left(\frac{B+1}{2},1\right).
\]
Then $2r_1\ell_1-\ell_1^2=B$.

If $B\equiv0\pmod4$, again one line bundle is enough: take
\[
(r_1,\ell_1)=\left(\frac B4+1,2\right).
\]
Then $2r_1\ell_1-\ell_1^2=B$.

Finally suppose that $B\equiv2\pmod4$. This is the only case in
which we use two line bundles. Put
\[
N=\frac{B+2}{2},
\]
so that $N$ is even. Choose
\[
\epsilon=
\begin{cases}
1,& A\equiv2\pmod4,\\
0,& A\not\equiv2\pmod4,
\end{cases}
\]
and set
\[
(r_1,\ell_1)=(\epsilon,1),
\qquad
(r_2,\ell_2)=(N-\epsilon,1).
\]
Then
\[
(2r_1\ell_1-\ell_1^2)
+
(2r_2\ell_2-\ell_2^2)
=
(2\epsilon-1)+(2N-2\epsilon-1)
=
B,
\]
so the two line bundles account for the entire $xy$-component.

It remains to check the residue of the coefficient
\[
C=A-r_1^2-r_2^2.
\]
Since $N$ is even,
\[
r_1^2+r_2^2
\equiv
\begin{cases}
0\pmod4,&\epsilon=0,\\
2\pmod4,&\epsilon=1.
\end{cases}
\]
Hence:
\[
\begin{array}{c|c}
A\pmod4 & C\pmod4\\
\hline
0&0\\
1&1\\
2&0\\
3&3
\end{array}
\]
where in the case $A\equiv2\pmod4$ we used $\epsilon=1$.
Therefore
$C\not\equiv2\pmod4.$
\end{proof}

\subsection{The residual bundle over \texorpdfstring{$\CP^2$}{CP2}}
When the residual rank is at least five, Proposition~\ref{prop:lifts} gives a commuting lift for its principal frame bundle over $\CP^2$. Only the rank-four case needs an additional check.

\begin{lemma}\label{lem:rank4-lift}
For every integer $C\not\equiv2\pmod4$, there exists an oriented real rank-four bundle
\[
\eta_C\longrightarrow\CP^2,
\qquad p_1(\eta_C)=Ch^2,
\]
whose principal $\SO(4)$ frame bundle admits a commuting lift of the action \eqref{eq:CP2-action}.
\end{lemma}

\begin{proof}
Let $\delta\in\{0,1\}$ be the parity of $C$, and choose the complex rank-two bundle $E_C$ from \eqref{eq:EC}:
\[
c_1(E_C)=\delta h,
\qquad
c_2(E_C)=\frac{\delta-C}{2}h^2.
\]
Put $\eta_C=(E_C)_{\R}$ with its complex orientation. Formula~\eqref{eq:realification} gives $p_1(\eta_C)=Ch^2$.

Suppose first that $C$ is odd. Then $\delta=1$, so
\[
w_2(\eta_C)=\rho_2(c_1(E_C))=\rho_2(h)\ne0.
\]
If the structure group of the frame bundle reduces to a torus, the required lift is already available by the torus-reduction discussion in Proposition~\ref{prop:lifts}. If it does not reduce to a torus, Grove--Ziller \cite[Corollary~4.11(b)]{GZ11} applies; its exceptional rank-four family is spin, so our nonspin bundle is not exceptional. Thus a commuting lift exists in either case.

Now suppose $C\equiv0\pmod4$. Since $C$ is even, the choice of $\delta$ gives $\delta=0$. Hence
\[
c_1(E_C)=0,
\qquad
c_2(E_C)=-\frac C2h^2,
\]
and therefore
\begin{equation}\label{eq:rank4-classes}
p_1(\eta_C)=Ch^2,
\qquad
e(\eta_C)=c_2(E_C)=-\frac C2h^2.
\end{equation}
For a non-torus-reducible rank-four bundle, \cite[Corollary~4.11(b)]{GZ11} gives the  characteristic classes in the form
\[
p_1=(2\kappa+2\lambda)h^2,
\qquad
e=(\kappa-\lambda)h^2,
\]
and the only possible exceptions occur when $\kappa$ and $\lambda$ are not both even. Comparing with \eqref{eq:rank4-classes} gives
\[
\kappa=0,
\qquad
\lambda=\frac C2.
\]
Because $C\equiv0\pmod4$, both integers are even. Hence the bundle lies outside the exceptional family and the commuting lift exists. If its structure group reduces to a torus, the lift again exists by the torus-reduction case. This proves the lemma.
\end{proof}

\begin{proposition}\label{prop:decomposition}
Let $\xi\to M$ have real rank $m\ge8$. Then there exist $d\in\{1,2\}$, complex line bundles $L_1,\ldots,L_d\to M$, and an oriented real rank-$(m-2d)$ bundle $\eta\to\CP^2$ such that
\begin{equation}\label{eq:decomposition}
\xi\cong
\bigoplus_{i=1}^{d}(L_i)_{\R}\oplus\pi^*\eta,
\end{equation}
and the principal frame bundle of $\eta$ admits a commuting $G$-lift.
\end{proposition}

\begin{proof}
Choose the integers $d,(r_i,\ell_i)$ and $C$ from Lemma~\ref{lem:arithmetic}, and put $L_i=L_{r_i,\ell_i}$. Let
\[
r=m-2d.
\]
If $d=1$, then $r\ge6$. By Lemma~\ref{lem:four-dimensional-existence}, choose an oriented rank-$r$ bundle $\eta\to\CP^2$ with $p_1(\eta)=Ch^2$. Since $r\ge5$, Proposition~\ref{prop:lifts} gives a commuting lift to its frame bundle.

If $d=2$, then $r\ge4$ and Lemma~\ref{lem:arithmetic} ensures $C\not\equiv2\pmod4$. Choose the rank-four bundle from Lemma~\ref{lem:rank4-lift}; if $r>4$, add $r-4$ trivial real lines. Extension of the structure group preserves the commuting lift.

In either case, \eqref{eq:line-p1} and \eqref{eq:arithmetic} show that the right-hand side of \eqref{eq:decomposition} has first Pontryagin class
\[
A x^2+B\,xy=p_1(\xi).
\]
Its rank is $m$. Corollary~\ref{cor:M-classification} therefore identifies it with $\xi$ as an actual rank-$m$ bundle.
\end{proof}

\subsection{Proof of Theorem~\ref{thm:blowup}}
Choose a decomposition \eqref{eq:decomposition}. Write $r=\rank\eta$, and let
\[
\SO(r)\longrightarrow P_\eta\xrightarrow{p_\eta}\CP^2
\]
be its oriented frame bundle. We now realize all summands of \eqref{eq:decomposition} simultaneously from one compact nonnegatively curved principal bundle.

\medskip
\noindent\emph{Step 1: a nonnegatively curved bundle over $\CP^2$.}
Let $q_0:S^5\to\CP^2$ be the complex Hopf bundle and form
\begin{equation}\label{eq:Qblowup}
Q=S^5\times_{\CP^2}P_\eta.
\end{equation}
The same $G=\SU(2)$-action on $\CP^2$ admits commuting lifts to the Hopf bundle and to $P_\eta$. Lemma~\ref{lem:fiber-product} therefore gives $Q$ a metric with $\secc\ge0$ invariant under its right structure group $S^1\times\SO(r)$.

\medskip
\noindent\emph{Step 2: replace the base $\CP^2$ by $M$.}
Let $Z=S^5\times S^3\to M$ be the principal $T^2$-bundle from \eqref{eq:torus-base}.  Using the composite $Z\to M\xrightarrow{\pi}\CP^2$, form the pullback
\begin{equation}\label{eq:Qhat-blowup}
\widehat Q=Z\times_{\CP^2}P_\eta.
\end{equation}
Because $Z\cong S^5\times S^3$ and the map to $\CP^2$ is the Hopf projection on the $S^5$-factor, there is a natural identification
\[
\widehat Q\cong Q\times S^3.
\]
Give $S^3$ its standard metric.  The product metric on $\widehat Q$ has nonnegative sectional curvature.  Moreover, the $T^2$-action in the Escher--Ziller model is induced by their unitary $U(2)$-action on $S^5\times S^3$ \cite[Proposition~5.5]{EZ14}; on the $S^5$-factor it is along the Hopf fibers, and on the $S^3$-factor it is isometric for the round metric.  Since the metric on $Q$ is invariant under the Hopf circle and under $\SO(r)$, the product metric is invariant under
\[
K=T^2\times\SO(r).
\]
The group $K$ acts freely on \eqref{eq:Qhat-blowup}, with $T^2$ acting on $Z$ and $\SO(r)$ acting on $P_\eta$.  Quotienting by the two factors gives
\begin{equation}\label{eq:hatQ-base}
\widehat Q/K
\cong Z/T^2
\cong M.
\end{equation}
Thus $\widehat Q\to M$ is a compact principal $K$-bundle whose total space has nonnegative sectional curvature.

\medskip
\noindent\emph{Step 3: recover the prescribed vector bundle.}
For each line bundle $L_i=L_{r_i,\ell_i}$ in \eqref{eq:decomposition}, let $\chi_i=\chi_{r_i,\ell_i}$ be the character supplied by Lemma~\ref{lem:characters}.  Consider the orthogonal real $K$-representation
\begin{equation}\label{eq:blowup-representation}
V=\bigoplus_{i=1}^{d}\C_{\chi_i}\oplus\R^r,
\end{equation}
where $T^2$ acts on the complex lines through the characters and trivially on $\R^r$, while $\SO(r)$ acts trivially on the complex lines and in the standard way on $\R^r$.

For a complex summand, the $\SO(r)$-factor acts trivially, so
\[
\widehat Q\times_K\C_{\chi_i}
\cong Z\times_{T^2}\C_{\chi_i}
=L_i.
\]
For the real summand, the $T^2$-factor acts trivially on $\R^r$.  Quotienting first by $T^2$ identifies
\[
\widehat Q/T^2
\cong M\times_{\CP^2}P_\eta,
\]
and hence
\[
\widehat Q\times_K\R^r
\cong
\bigl(M\times_{\CP^2}P_\eta\bigr)\times_{\SO(r)}\R^r
\cong\pi^*\eta.
\]
Associated bundles preserve direct sums.  Therefore
\begin{equation}\label{eq:blowup-associated-identification}
\widehat Q\times_KV
\cong
\bigoplus_{i=1}^{d}(L_i)_{\R}\oplus\pi^*\eta
\cong\xi.
\end{equation}

\medskip
\noindent\emph{Step 4: pass the curvature to the associated bundle.}
Give $V$ its $K$-invariant Euclidean inner product. Lemma~\ref{lem:quotient-curvature}, applied to the compact nonnegatively curved principal bundle $\widehat Q\to M$, gives a complete metric with nonnegative sectional curvature on
\[
\Tot(\xi)\cong\widehat Q\times_KV.
\]
For the sphere bundle, use the unit sphere $S(V)\subset V$ with its standard round metric. The same lemma gives
\[
S(\xi)\cong\widehat Q\times_KS(V)
\]
a metric of nonnegative sectional curvature. This completes the proof of Theorem~\ref{thm:blowup}.

\section{The rank bounds}

The two bounds reflect different parts of the argument. Over $\CP^3$, classification by $p_1$ already holds in rank seven. The odd-parity model then has an auxiliary rank-five bundle over $S^4$, so all of the required constructions are available. In rank six, the Euler class supplies additional unstable information \cite[Proposition~1]{CV92}; matching $p_1$ alone no longer identifies the desired bundle. This is not a failure of the Grove--Ziller lift for rank-four bundles over $S^4$.

Over $M$, classification by $p_1$ likewise holds in rank seven. A decomposition using two line bundles and an arbitrary rank-at-least-five bundle over $\CP^2$ would give the bound nine immediately. Lemma~\ref{lem:arithmetic} improves this: when two line bundles are needed, their weights can be chosen so that the residual integer $C$ is odd or divisible by four. Lemma~\ref{lem:rank4-lift} then supplies a liftable rank-four bundle, giving the bound eight. Thus the existence of exceptional rank-four lifts alone does not justify stopping at nine.

The proof does not settle all rank-seven bundles over $M$. After two realified line summands, the auxiliary bundle would have rank three, where the admissible characteristic classes and commuting-lift conditions are more restrictive \cite[Theorem~C and Corollary~4.11(a)]{GZ11}. We make no claim that either lower-rank conclusion is false. The fixed-rank metric constructions above require neither curvature-preserving cancellation nor an extension of the classification by $p_1$ beyond its stated range.

\end{document}